\documentclass[11pt]{article}
\usepackage[utf8]{inputenc}

\usepackage{geometry} 
\usepackage{graphicx} 

\usepackage{booktabs} 
\usepackage{array} 
\usepackage{paralist} 
\usepackage{verbatim} 
\usepackage{subfig} 
\usepackage{hyperref}
\usepackage{amsmath}
\usepackage{amssymb}
\usepackage{authblk}
\usepackage{stmaryrd}
\usepackage{bbm}
\usepackage{mathtools}
\usepackage{textcomp}
\usepackage{xcolor}
\usepackage{listings}
\usepackage{float}
\usepackage{centernot}
\usepackage[mathscr]{euscript}
\usepackage{tikz-cd}

\newcommand{\Ccal}{\mathcal{C}}
\newcommand{\Dcal}{\mathcal{D}}
\newcommand{\Ecal}{\mathcal{E}}
\newcommand{\Fcal}{\mathcal{F}}

\newcommand{\Ocal}{\mathcal{O}}

\newcommand{\Set}{\mathbf{Set}}
\newcommand{\Pos}{\mathbf{Pos}}
\newcommand{\Mon}{\mathbf{Mon}}
\newcommand{\Grp}{\mathbf{Grp}}

\newcommand{\Hom}{\mathrm{Hom}}
\newcommand{\Nat}{\mathrm{Nat}}
\newcommand{\Cat}{\mathbf{Cat}}

\newcommand{\TOP}{\mathfrak{TOP}}
\newcommand{\op}{^{\mathrm{op}}}
\newcommand{\co}{^{\mathrm{co}}}

\newcommand{\too}{\twoheadrightarrow}

\DeclareMathOperator{\id}{id}

\DeclareMathOperator{\Sh}{Sh}

\DeclareMathOperator{\End}{End}

\tikzset{
  no line/.style={draw=none,
    commutative diagrams/every label/.append style={/tikz/auto=false}},
  from/.style args={#1 to #2}{to path={(#1)--(#2)\tikztonodes}}
	}
\tikzset{symbol/.style={draw=none, every to/.append style={edge node = {node [sloped, allow upside down, auto=false] {$#1$}}}}}

\usepackage{amsthm}
\newtheorem{thm}{Theorem}[section]
\newtheorem{prop}[thm]{Proposition}
\newtheorem{lemma}[thm]{Lemma}
\newtheorem{crly}[thm]{Corollary}

\theoremstyle{definition}
\newtheorem{dfn}[thm]{Definition}
\newtheorem{xmpl}[thm]{Example}
\theoremstyle{remark}

\usepackage{fancyhdr} 
\usepackage{sectsty}
\allsectionsfont{\sffamily\mdseries\upshape} 

\usepackage[nottoc,notlof,notlot]{tocbibind} 
\usepackage[titles,subfigure]{tocloft} 

\usepackage[nodayofweek]{datetime}
\longdate

\date{\today} 

\title{Toposes of Discrete Monoid Actions}
\author[1]{Morgan Rogers \thanks{Universit\`a degli Studi dell{'}Insubria, Via Valleggio n. 11, 22100 Como CO} \thanks{Marie Sklodowska-Curie fellow of the Istituto Nazionale di Alta Matematica} \thanks{email: mrogers@uninsubria.it}}

\begin{document}

\maketitle{}

\abstract{Properties of toposes of right $M$-sets are studied, and these toposes are characterised up to equivalence by their canonical points. The solution to the corresponding Morita equivalence problem is presented in the form of an equivalence between a 2-category of monoids and the corresponding 2-category of toposes.}


\section{Introduction}

The most easily described examples of Grothendieck toposes are presheaf toposes: categories of the form $[\Ccal\op,\Set]$, whose objects are contravariant functors from a small category $\Ccal$ to the category of sets and whose morphisms are natural transformations between these. A category theorist trying to understand Grothendieck toposes for the first time will therefore naturally ask `what does such a category look like for the simplest choices of $\Ccal$'? In particular, what happens when $\Ccal$ is a preorder or a monoid?

Due to the origins of topos theory from categories of sheaves on topological spaces, where the corresponding sites are frames (that is, very structured preorders), one side of this question has been treated far more thoroughly in the topos-theory literature than the other. Toposes of the form $[M\op,\Set]$ do feature as illustrative examples in introductory texts such as \cite{MLM} and \cite{TTT}. However, if one looks to the most comprehensive topos theory references to date, notably P.T. Johnstone's work \cite{Ele}, there is no systematic treatment of monoids (or even of groups, beyond some examples) which parallels the one for locales.

On the other hand, these toposes \textit{have} been studied by semigroup theorists and ring theorists as an extension of the representation theory of rings. For example, Knauer in \cite{MEM} and Banaschewski in \cite{FCM} independently solved the Morita equivalence problem for (left) actions of discrete monoids. In this context, the category of presheaves on a monoid $M$ is better thought of as the category of \textbf{monoid actions} or \textbf{$M$-sets}, since it consists of the collection of (right) actions of the monoid on sets. Their results subsequently featured in a reference text \cite{MAC} on categories of monoid actions, published early enough that the word `topos' does not appear in the work.

Topos theory provides a broader perspective from which these problems can be resolved much more efficiently, but conversely if it can be shown that a given Grothendieck topos is equivalent to one constructed from a monoid, there is immediate access to extensive algebraic results from semigroup theory. More interestingly, extracting topos-theoretic invariants corresponding to properties of monoids can provide tools within topos theory complementary to the locale-centric ones that currently dominate the literature. These tools will be valuable for the `toposes as bridges' research programme proposed by O. Caramello in \cite{TST}.

It should be mentioned that toposes associated to groups arise in topos-theoretic treatments of Galois theory such as \cite{LGT} or \cite{TGT}, and some such results have been extended to the more general context of toposes associated to monoids. Notably, \cite{SGC} studies the actions of pro-finite topological monoids. It is the author's hope that a systematic treatment of toposes associated to monoids (toward which this article constitutes an initial contribution) will yield further insight into these cases. 

In this article we present a characterization of categories $\Ecal$ which are equivalent to a category $[M\op,\Set]$ of right $M$-sets for some monoid $M$. We also present a more categorical route to the solution of the corresponding Morita equivalence problem for monoids: the question of whether $M$ is uniquely defined by the topos of right $M$-sets, or if non-isomorphic monoids can have equivalent categories of presheaves. As such, we show how the results of Knauer and Banaschewski on this subject can be derived in topos-theoretic language.

The main results are Theorem \ref{thm:point} which characterises toposes of the form $[M\op,\Set]$, Theorem \ref{thm:2equiv} which functorialises this characterisation, and Corollary \ref{crly:Morita} which establishes the results about Morita equivalence. The article is written to be mostly self-contained, introducing relevant topos-theoretic terms and properties for the benefit of readers from outside topos theory.

This work was supported by INdAM and the Marie Sklodowska-Curie Actions as a part of the \textit{INdAM Doctoral Programme in Mathematics and/or Applications Cofunded by Marie Sklodowska-Curie Actions}. The author would like to thank Olivia Caramello for her essential guidance and suggestions.

\section{Monoids, their Idempotent Completions and their Presheaves}
\label{sec:disc}

For the purposes of this investigation, we treat a monoid $M$ as a (small!) category with a single object; the identity shall be denoted $1$. The analysis takes place at three levels: the level of the monoids themselves, the level of their associated presheaf toposes, and the intermediate level of their idempotent completions.

Even a priori, these considerations can easily be extended to semigroups, since any semigroup $S$ has a category of right $S$-sets (to be precise, the category of sets $X$ equipped with a semigroup homomorphism $S\op \to \End(X)$). By freely adding an identity element to $S$, it becomes a monoid $S_1$ such that $[{S_1}\op,\Set]$ is equivalent to the category of $S$-sets, since a monoid homomorphism $S_1 \to \End(X)$ necessarily sends the new identity to the identity of $\End(X)$, and is therefore determined by a semigroup homomorphism $S \to \End(X)$. It follows that for the purposes of a classification of toposes of this form there is no difference. However, we shall show later that semigroup homomorphisms, rather than monoid homomorphisms, are the right morphisms to consider in order to capture more information about geometric morphisms between toposes and to properly describe Morita equivalence.

Recall that a category $\Ccal$ is \textbf{idempotent complete} (or Cauchy or Karoubi complete) if all idempotent morphisms in $\Ccal$ split. Recall also that any given category $\Ccal$ has an \textbf{idempotent completion}, denoted $\check{\Ccal}$, equipped with a full and faithful functor $\Ccal \to \check{\Ccal}$ universal amongst functors from $\Ccal$ to idempotent complete categories. For a more detailed reminder and a construction of the idempotent completion in general, see the discussion in \cite{Ele} which begins just before Lemma A1.1.8.

For a monoid $M$, $\check{M}$ can be identified up to equivalence with a category whose objects are idempotents of $M$, and this is the definition of $\check{M}$ we shall use since the resulting idempotent splittings in this category are uniquely defined. Where necessary for clarity, we shall denote by $\underline{e}$ the object of $\check{M}$ corresponding to an idempotent $e$. The morphisms $\underline{e} \to \underline{d}$ in this category are morphisms $f$ of $M$ such that $fe = f = df$; composition is inherited from $M$. $M$ is included in $\check{M}$ as the full subcategory on the object $\underline{1}$. 

\begin{dfn}
Recall that an object $C$ of a category $\Ccal$ is \textbf{projective} if whenever there exists a morphism $f: C \to B$ and an epimorphism $g:A \too B$, there is a lifting $f':C \to A$ with $f = gf'$.

An object $C$ is \textbf{indecomposable} (or \textbf{connected}) if $C$ is not initial and whenever $C \cong A \sqcup B$, one of the coproduct inclusions is an isomorphism.
\end{dfn}

To justify the introduction of idempotent completions, we point to the lemmas \cite{Ele}[A1.1.9, A1.1.10] and their natural corollary:
\begin{lemma}
\label{lem:idempotent}
For any category $\Ccal$, $[\Ccal\op,\Set] \simeq [\check{\Ccal}\op,\Set]$, and $\check{\Ccal}$ is equivalent to the full subcategory of $[\Ccal\op,\Set]$ whose objects are the indecomposable projectives. Thus $[\Ccal\op,\Set] \simeq [\Dcal\op, \Set]$ if and only if $\check{\Ccal} \simeq \check{\Dcal}$
\end{lemma}

Thus $[M\op,\Set] \simeq [{M'}\op,\Set]$ if and only if $\check{M} \simeq \check{M'}$. Since it is easily shown that $(M\op)^{\vee} \simeq \check{M}\op$, this immediately gives a result which is not at all obvious from the algebraic description of the category of $M$-sets:
\begin{crly}
$[M\op,\Set] \simeq [{M'}\op,\Set]$ if and only if $[M,\Set] \simeq [M',\Set]$; there is no need to distinguish between `left' and `right' Morita equivalence of monoids.
\end{crly}

Before presenting the definitive solution to the question of Morita equivalence, we shall exhibit some properties of toposes of presheaves on monoids.

\section{Properties of Toposes of Presheaves on Monoids}

Recall that the forgetful functor $U:[M\op, \Set] \to \Set$ sending a right $M$-set to its underlying set is both monadic and comonadic. In particular, it has left and right adjoints,
\[\begin{tikzcd}
\Set \ar[r, "(-) \times M", bend left = 50] \ar[r, "\Hom_{\Set}(M{,}-)"', bend right= 50] \ar[r, symbol = \bot, shift right = 5, near end] \ar[r, symbol = \bot, shift left = 5, near end] & {[M\op, \Set]} \ar[l, "U"],
\end{tikzcd}\]
where the action of $M$ on $X \times M$ is simply multiplication on the right and the action of $m \in M$ on $\Hom_{\Set}(M,-)$ sends $f \in \Hom_{\Set}(M,X)$ to $f\cdot m := (x \mapsto f(mx))$.

Monadicity is intuitive, since $[M\op, \Set]$ is easily seen to be (equivalent to) the category of algebras for the free-forgetful adjunction: an algebra is a set $A$ equipped with a morphism $A \times M \to A$ satisfying identities that correspond to those for an $M$-action.

\begin{dfn}
Recall that for toposes $\Ecal$ and $\Fcal$, a \textbf{geometric morphism} $\phi:\Ecal \to \Fcal$ consists of a functor $\phi_*: \Ecal \to \Fcal$ called the \textbf{direct image functor}, admitting a left adjoint $\phi^*: \Fcal \to \Ecal$ called the \textit{inverse image functor} which preserves finite limits.

A geometric morphism is \textbf{essential} if $\phi^*$ admits a further left adjoint, $\phi_!$. A \textbf{point} of a Grothendieck topos $\Ecal$ is simply a geometric morphism $\Set \to \Ecal$. Finally, a geometric morphism is \textit{surjective} if its inverse image functor is comonadic.
\end{dfn}

Therefore from a topos-theoretic perspective, $U$ is the inverse image of an \textbf{essential surjective point} of $[M\op,\Set]$; this is the first property of note. We shall call this point the \textbf{canonical point} of $[M\op,\Set]$, although we emphasise that the canonicity is relative to the representation by $M$; a priori there may be other representations with corresponding canonical points.

Next, note that the terminal object $1$ of $[M\op,\Set]$ is the trivial action of $M$ on the one-element set. In particular, the only subobjects of $1$ (the \textbf{subterminal objects}) are itself and the empty $M$-action, which is to say that $[M\op, \Set]$ is \textbf{two-valued}. This property immediately gives:

\begin{lemma}
\label{lem:orthog}
For a locale $X$, the localic topos $\Sh(X)$ is equivalent to a topos of the form $[M\op,\Set]$ if and only if both $X$ and $M$ are trivial. Similarly, for any preorder $P$, $[P\op,\Set] \simeq [M\op,\Set]$ if and only if $P$ is equivalent to the one-element poset and $M$ is trivial.
\end{lemma}
\begin{proof}
The frame $\Ocal(X)$ of the locale $X$ is isomorphic to the frame of subterminal objects of $\Sh(X)$, but for any $M$, $[M\op,\Set]$ is two-valued, so $\Ocal(X)$ is the initial frame, making $X$ the terminal locale, so $\Sh(X) \simeq \Set$. There is a unique geometric morphism $\Set \to \Set$ which must coincide with the canonical point described above, but the induced comonad therefore sends any object $A$ to $\Hom_{\Set}(M,A) \cong A$, which forces $M$ to have exactly one element and hence be trivial.

The subterminal objects of $[P\op,\Set]$ can be identified with the downward closed sets, and it is easily seen that if any element is not a top element, the principal downset generated by that element gives a non-trivial subterminal object, and the topos fails to be two-valued; it follows that to be two valued, every element of $P$ must be a top element (and $P$ must be non-empty), which gives an equivalence with the one-element poset. The remainder of the argument is as above.
\end{proof}

Lemma \ref{lem:orthog} illustrates that the conceptual `orthogonality' between preorders and monoids as contrasting families of small categories extends in a concrete way to the topos-theoretic setting.

Observe that every right $M$-set $X$ can be expressed as a coproduct of indecomposable $M$-sets, which are precisely the equivalence classes of the equivalence relation generated by $x \sim y$ when $\exists m$ with $y = xm$. That is, $[M\op,\Set]$ has a separating set of indecomposable $M$-sets; when $M$ is a group, these are simply the orbits of the action. By \cite{SCCT}[Theorem 2.7], this makes $[M\op,\Set]$ a \textbf{locally connected topos}, with corresponding site consisting of the full subcategory on the indecomposable objects. Equivalently (over $\Set$), the unique geometric morphism $[M\op,\Set] \to \Set$ is essential.

More directly, we can compute that the unique geometric morphism to $\Set$ is:
\[\begin{tikzcd}
{[M\op, \Set]} \ar[r, "C", bend left = 50] \ar[r, "\Gamma"', bend right= 50]
\ar[r, symbol = \bot, shift right = 5, near start] \ar[r, symbol = \bot, shift left = 5, near start] &
\Set \ar[l, "\Delta"],
\end{tikzcd}\]
where $\Gamma$ sends an $M$-set to the subset of its elements on which $M$ acts trivially, $\Delta$ sends a set $A$ to the coproduct $\coprod_{a \in A} 1$ of copies of the terminal $M$-set and $C$ sends $X$ to its set of orbits.

Finally, the geometric morphism to $\Set$ is \textbf{hyperconnected}: the inverse image is full and faithful and its image is closed under subquotients. This can be deduced from the fact that the defining presentation of $[M\op,\Set]$ is as presheaves on a strongly connected category (a category in which there is at least one morphism $A \to B$ for every ordered pair of objects $A,B$); see the discussion following \cite{Ele}[A4.6.9].

\section{Essential Points of Presheaves on a Monoid}

The first problem, given a topos of the form $[M\op,\Set]$, is to identify whether $M$, or at least some presenting monoid, can be recovered from the structure of the topos. By the Yoneda Lemma, we know that $M$ is the full subcategory on the representable corresponding to its unique object, which is irreducible and projective; indeed, it is precisely $M$ viewed as a right $M$-set. Does every irreducible projective give a valid representation of the topos?

In \cite{TT}[Ex. 7.3], one can find the following general result for an arbitrary Grothendieck topos $\Ecal$ (in fact it is stated there for a base topos possibly distinct from $\Set$), which provides a connection between essential points and irreducible projective objects. The source cited there is not especially accessible, so we reprove it here.

\begin{lemma}
A functor $\phi: \Ecal \to \Set$ is the inverse image of an essential point if and only if it has the form $\Hom_{\Ecal}(Q,-)$ for $Q$ a projective indecomposable object.
\label{lem:proj}
\end{lemma}
\begin{proof}
First, we observe that $\phi$ has a left adjoint if and only if it is representable. If $\phi = \Hom(Q,-)$, then $\phi$ certainly preserves all limits by their universal properties, so it has a left adjoint by the special adjoint functor theorem, say. Conversely, if $\phi$ has a left adjoint $\phi_!$, then for an object $E$ of $\Ecal$, it must be that $\phi(E) \cong \Hom_\Set(1,\phi(E)) \cong \Hom_\Ecal(\phi_!(1),E)$, so $\phi$ is represented by $Q:=\phi_!(1)$. Indeed, it follows that $\phi_!(A) = \coprod_{a \in A}Q$.

To demonstrate check the existence of the right adjoint, we invoke the special adjoint functor theorem, which given cocompleteness of toposes states that it suffices to check preservation of colimits.

Since the initial object is strict in a topos, $\Hom_\Ecal(Q,0) = \emptyset$ holds if and only if $Q \not\cong 0$.

To preserve coproducts, it is required that $\Hom_\Ecal(Q,\coprod_{i\in I} A_i) = \coprod_{i\in I}\Hom_\Ecal(Q,A_i)$; that is, every arrow from $Q$ to a coproduct must factor uniquely through one of the coproduct inclusions. If this is so and $Q \cong Q_1 \sqcup Q_2$ then the identity on $Q$ without loss of generality factors through the inclusion of $Q_1$, and since coproducts are disjoint in $\Ecal$, this forces $Q_2 \cong 0$, so $Q$ is indecomposable. Conversely, if $Q$ is indecomposable and $f \in \Hom_\Ecal(Q,\coprod_{i\in I} A_i)$ then consider $B_i = f^*(A_i)$. Since coproducts are stable under pullback, these form disjoint subobjects of $Q$ and $Q \cong \coprod_{i \in I} B_i$. Indecomposability of $Q$ forces $B_i \cong Q$ for some $i$, and hence one can uniquely identify $f$ with a member of $\Hom_\Ecal(Q,A_i)$.

Finally, $Q$ being projective is equivalent to $\Hom(Q,-)$ preserving epis, which we claim is equivalent to preserving coequalizers given the preservation of coproducts.

All epis in $\Ecal$ are regular, so preservation of coequalizers certainly implies preservation of epimorphisms. Conversely, given a parallel pair $f,g: A \rightrightarrows B$ in $\Ecal$, consider its factorization through the kernel pair of its coequalizer:
\[\begin{tikzcd}
A \ar[r, "\exists ! e"] & B' \ar[r, shift left, "f'"] \ar[r, shift right, "g'"'] & B \ar[r, "c", two heads] & C.
\end{tikzcd}\]
$\Hom(Q,-)$ preserving epis and monos ensures that it preserves image factorizations, so without loss of generality $R = \langle f,g \rangle$ is a relation on $B$ (else take its image in $B \times B$). For $n > 1$, $R^n$ is computed via pullbacks and images, so is also preserved by $\Hom(Q,-)$, as is the diagonal subobject $R^0$. Now, $c$ is precisely the quotient of $B$ by the equivalence relation generated by $R$, which is computed as the image of the coproduct of $R^n$ for $n \geq 0$, also preserved. Hence the coequalizer of $\Hom(Q,f)$ and $\Hom(Q,g)$ is the quotient of $\Hom(Q,B)$ by the generated equivalence relation, and is precisely $\Hom(Q,C)$. We conclude that $\Hom(Q,-)$ preserves all coequalizers.
\end{proof}

Considering the construction of $\check{M}$ described earlier and Lemma \ref{lem:idempotent}, it follows that:

\begin{crly}
\label{crly:idem}
The essential points of $[M\op, \Set]$ correspond precisely to its idempotents, via the correspondences:
\begin{align*}
\{\text{idempotents}\}
&\leftrightarrow \{\text{objects of }\check{M}\} \\
&\leftrightarrow \{\text{non-empty indecomposable projectives in }[M\op,\Set]\}/(\text{isomorphism}) \\
&\leftrightarrow \{\text{essential points of }[M\op,\Set]\}
\end{align*}
\end{crly}

While Corollary \ref{crly:idem} shows that there are typically many essential points of $[M\op,\Set]$, not every such is a candidate for an essential \textit{surjective} point. We return to the more general setting briefly.

\begin{lemma}
Let $\phi$ be the essential point of a Grothendieck topos $\Ecal$ induced by an indecomposable projective object $Q$. Then the following are equivalent:
\begin{enumerate}
	\item $\phi^*$ is comonadic (equivalently, $\phi$ is surjective).
	\item $\phi^*$ is faithful.
	\item $\phi^*$ is conservative.
	\item $Q$ is a \textbf{separator} (also referred to as a \textbf{generator}).
	\item $\phi^*$ is monadic.
\end{enumerate}
\label{lem:monad}
\end{lemma}
\begin{proof}
1 $\Leftrightarrow$ 2 $\Leftrightarrow$ 3 is a special case of Lemma A4.2.6 in \cite{Ele}. Being faithful, $\phi^*$ reflects monos and epis. Since $\Ecal$ is balanced, this is sufficient to reflect isomorphisms. 

3 $\Leftrightarrow$ 4 Recall that $\phi^* = \Hom(Q,-)$. Since every topos has equalizers and is balanced, $Q$ is a separator if and only if it detects isomorphisms (see \cite{PTJCT} Lemma 2.19), which is immediately equivalent to $\Hom(Q,-)$ reflecting isomorphisms.

3 $\Leftrightarrow$ 5 Certainly $\Ecal$ has and $\phi^*$ preserves coequalizers of $\phi^*$-split pairs (and even coequalizers of reflexive pairs), since it has a left and right adjoint. Thus $\phi^*$ is monadic by Beck's monadicity theorem if and only if it is conservative.
\end{proof}

Applied to $[M\op,\Set]$, the statement that the object $Q$ corresponding to the canonical point should be a separator is not especially surprising, since the objects of a topos coming from a site representing it always form a separating family, and in this instance there is just one object. More generally, we find that one-object separating families are related very strongly to one another.

\begin{lemma}
In an infinitary extensive, locally small category (and in particular in any Grothendieck topos) any pair of indecomposable projective separators are retracts of one another, and conversely if $Q, Q'$ are retracts of one another and $Q$ is an indecomposable projective separator, so is $Q'$.
\label{lem:retracts}
\end{lemma}
\begin{proof}
Let $\Ccal$ be an extensive category and suppose $Q, Q'$ are indecomposable projective separators. First, note that for any object $A$ the two coproduct injections $\iota_1, \iota_2 : A \rightrightarrows A \sqcup A$ are equal if and only if $A \cong 0$, since their equalizer factors through the pullback of one against the other, which is $0$ since coproducts are disjoint.

Thus since $Q$ is a separator and $Q' \not\cong 0$, there is at least one morphism $Q \to Q'$ to distinguish its coproduct injections into $Q \sqcup Q$. Moreover, the collection of all morphisms $Q \to Q'$ is jointly epic, which is to say that the composite morphism $\coprod Q \too Q'$ is epic. Since $Q'$ is projective, this epimorphism splits; there is some $Q' \hookrightarrow \coprod Q$. But $Q'$ being indecomposable forces this morphism to factor through one of the coproduct inclusions, making $Q'$ a retract of $Q$. A symmetric argument makes $Q$ a retract of $Q'$.

Now suppose $Q,Q'$ are retracts of one another and $Q$ is an irreducible projective separator. $Q'$ is projective since any retract of a projective object is, $Q'$ is not initial since it admits a morphism from $Q$, $Q'$ is indecomposable since pulling back any coproduct decomposition along the epi from $Q$ forces all but one of the components to be $0$, and $Q'$ is a separator since $\Hom(Q',-)$ surjects onto $\Hom(Q,-)$ by composition with the epi $Q' \too Q$, so the former functor is conservative when the latter is.
\end{proof}

More intuitively, in $\check{M}$, each object $\underline{e}$ is a retract of $\underline{1}$ (via the morphisms indexed by $e$). Any other candidate for a monoid which generates the same idempotent complete category must be the monoid of endomorphisms of one of these objects, and thus $\underline{1}$ must be a retract of the corresponding idempotent. This can be used directly to derive Corollary \ref{crly:Morita}, and indeed Banaschewski proceeds with this argument in \cite{FCM}. However, it is more convenient to reach the characterisation via the main new result of this paper.

\section{Characterisation of Presheaves on Monoids}

Now that we have established strong constraints on the candidates for surjective essential points of any topos, we show in this section that any such point gives a canonical representation of the topos as the category of presheaves on a monoid.

Given an indecomposable projective separator $Q$, $\phi^* = \Hom_{\Ecal}(Q,-)$ has left adjoint $\phi_! : \Set \to \Ecal$ given by $\phi_!(A) = \coprod_{a \in A} Q$, since $\phi_!$ must preserve coproducts and $\phi_!(1) \cong Q$ from the proof of Lemma \ref{lem:proj}.

\begin{lemma}
\label{lem:monadoid}
Let $\Phi := \phi^* \phi_!$ be the functor part of the monad induced by the essential surjective point $\phi$ as above. Then $\Phi (1) = \phi^*(Q) = \Hom_{\Ecal}(Q,Q)$. Moreover, $\Phi^2(1) \cong \Phi(1) \times \Phi(1)$, and the unit and multiplication morphisms make $\Phi(1)$ into a monoid.
\end{lemma}
\begin{proof}
Since $Q$ is an indecomposable projective, $\Hom(Q,-)$ preserves coproducts, so $\Phi^2(1) \cong \coprod_{f \in \Hom(Q,Q)}\Hom(Q,Q)$, which is of course isomorphic to $\Hom(Q,Q) \times \Hom(Q,Q)$ in $\Set$; by an identical argument, it follows that $\Phi^3(1) \cong \Phi(1)^3$. By direct computation, the multiplication sends $g$ in the copy of $\Hom(Q,Q)$ indexed by $f$ to $g \circ f$. The unit at the terminal object $\eta: 1 \to \Phi(1)$ picks out the identity morphism. The associativity and unit conditions follow from the identities satisfied by the monad.
\end{proof}

\begin{thm}
\label{thm:point}
Let $\Ecal$ be any category. The following are equivalent:
\begin{enumerate}
	\item $\Ecal$ is equivalent to $[M\op,\Set]$ for some monoid $M$.
	\item There exists a functor $\Ecal \to \Set$ which is monadic and comonadic.
	\item There exists a functor $\Ecal \to \Set$ which is monadic such that the free algebra on $1$ is indecomposable and projective.
	\item $\Ecal$ is a Grothendieck topos with at least one indecomposable projective separator.
	\item $\Ecal$ is a topos admitting an essential surjective point, $\Set \to \Ecal$.
\end{enumerate}
In particular, such an $M$ is recovered as the free algebra on the terminal object of $\Set$ for the monad $\Phi$ induced by the essential surjective point.
\end{thm}
\begin{proof}
Most of the proof is already established; the third point is a Corollary in \cite{FCM} whose equivalence to the fourth is established by observing that faithfulness of the functor to $\Set$ makes the free algebra a separator in $\Ecal$.

It remains to show that if $M:= \Phi(1)$ is the monoid obtained from the monad as above, then $\Ecal \simeq [M\op,\Set]$. Since $\phi^*$ is monadic, it suffices to identify the algebras of the monad with right $M$-sets. Indeed, $\Phi(A) = \Hom_{\Ecal}(Q,\coprod_{a \in A} Q) \cong \coprod_{a \in A} \Hom_{\Ecal}(Q,Q) \cong A \times M$, so an algebra structure map is a map $\alpha: A \times M \to A$ such that the identity in $M = \Hom_{\Ecal}(Q,Q)$ acts trivially and such that
\[\begin{tikzcd}
\Phi^2(A) \ar[r, "\Phi(\alpha)"] \ar[d, "\mu_{\Phi(A)}"'] &
\Phi(A) \ar[d, "\alpha"]\\
\Phi(A) \ar[r, "\alpha"] & A
\end{tikzcd}\]
commutes, which is to say that $a \cdot (g \circ f) = (a \cdot g) \cdot f$ for each $a \in A$, $f,g \in \Hom(Q,Q)$, so indeed the structure map makes $A$ a right $M$-set, as required. Conversely, the action map for a right $M$-set is clearly an algebra structure map, so the proof is complete.
\end{proof}

We should at this point thank Todd Trimble for a valuable discussion on MathOverflow and email in which he pointed out that any cocontinuous monad on $\Set$ must be of the form $(-)\times M$ for some set $M$, from which Lemma \ref{lem:monadoid} can easily be deduced. This certainly fails for cocontinuous monads over toposes in general, but nonetheless a similar argument to the above can in principle be used to recover some information about site representations from essential surjections; we give no further details here.

Before we conclude this section, we should noted that there is another approach to recovering $M$ from the canonical essential surjective point of $\Ecal \simeq [M\op,\Set]$ that is somewhat easier to generalise, variants of it having appeared in \cite{TGT} and \cite{LGT} to respectively recover topological and localic group representations of toposes from their points.

Since the inverse image functor of the point is representable, by the usual Yoneda argument there is an isomorphism of monoids:
\[\End(U) := \Nat(\Hom_{\Ecal}(Q,-),\Hom_{\Ecal}(Q,-)) \cong \Hom_{\Ecal}(Q,Q)\op \cong M\op\]
and hence this provides another way of recovering $M$.

\section{Morphisms Between Monoids and their Toposes}

A monoid homomorphism $f:M' \to M$ induces an essential geometric morphism $[{M'}\op,\Set] \to [M\op,\Set]$ whose inverse image in the restriction of $M$-actions along $M'$. This morphism is always a surjection, being induced by a functor which is surjective on objects (see \cite{Ele}[A4.2.7(b)]). Notably, the canonical points studied in Section \ref{sec:disc} are induced by the inclusion of the trivial monoid into a given monoid $M$. This is not the only possible source of essential surjections, since any equivalence is an essential surjection and as we shall show below that not every equivalence is induced by a monoid homomorphism.

On the other hand, if $e$ is an idempotent of $M$, it clear that $eMe := \{eme \mid m \in M\}$ equipped with the restricted multiplication operation is a monoid with identity $e$.

\begin{lemma}
\label{lem:includes}
Each (semigroup homomorphism) inclusion of $M' = eMe$ into $M$ produces a fully faithful inclusion $\check{M}' \hookrightarrow \check{M}$ of the respective idempotent completions. Hence the induced essential geometric morphism $[{eMe}\op,\Set] \to [M\op,\Set]$ is an \textbf{inclusion} (its direct image is full and faithful).
\end{lemma}
\begin{proof}
Observe that $M' = eMe$ consists precisely of those elements $m \in M$ such that $eme = m$; in particular the idempotents of $M'$ are indexed by idempotents $f \in M$ with $ef = f = fe$. In the idempotent completion $\check{M}$, recall that the morphisms $\underline{f} \to \underline{f}'$ (with $f,f' \in eMe$) are those $m \in M$ such that $f'mf = m$. But then $eme = ef'mfe = f'mf = m$. Hence $m$ lies in $M'$ and $\check{M}'$ is precisely the full subcategory of $\check{M}$ on the objects corresponding to the idempotents $f$ with $ef = f = fe$.

The proof that this makes the resulting geometric morphism an inclusion is described in \cite{Ele}[A4.2.12(b)].
\end{proof}

More generally, \textit{any} semigroup homomorphisms $f:M' \to M$ factors canonically as a monoid homomorphism to $f(1)Mf(1)$ followed by an inclusion of the above form. The equivalence in Theorem \ref{thm:2equiv} below lifts this canonical factorization to the topos level, where it is a special case of the surjection-inclusion factorization of geometric morphisms described in \cite{Ele}[A4.2.10].

\begin{dfn}
Let $f,g:M' \to M$ be semigroup homomorphisms. A \textbf{conjugation}\footnote{This is the author's own terminology.} $\alpha$ from $f$ to $g$, denoted $\alpha:f \Rightarrow g$ is an element $\alpha \in M$ such that $\alpha f(1') = \alpha = g(1') \alpha$ and for every $m' \in M'$, $\alpha f(m') = g(m') \alpha$. The conjugation $\alpha$ is said to be \textbf{invertible} if there exists a conjugation $\alpha': g \Rightarrow f$ with $\alpha' \alpha = f(1)$ ant $\alpha \alpha' = g(1)$; note that $\alpha$ need not be a unit of $M$ to be invertible as a conjugation.
\end{dfn}

\begin{prop}
\label{prop:extend}
Let $M, M'$ be monoids. Then functors $\check{f},\check{g}:\check{M}' \to \check{M}$ correspond uniquely to semigroup homomorphisms $f,g: M' \to M$, and any natural transformation $\check{\alpha}: \check{f} \rightarrow \check{g}$ is determined by the conjugation $\alpha = \check{\alpha}_{1'}:f \Rightarrow g$. A conjugation is invertible if and only if it corresponds to a natural isomorphism.
\end{prop}
\begin{proof}
Of course, $f$ is the restriction of $\check{f}$ to $M'$ (that is, to the full subcategory on $\underline{1}'$). This produces a semigroup homomorphism $M' \to M$, since it gives a monoid homomorphism from $M'$ to $eMe$, where $e$ is the idempotent such that $\underline{e} = f(\underline{1}')$; this monoid then includes into $M$ via a semigroup homomorphism as in Lemma \ref{lem:includes}.

Conversely, any semigroup homomorphism $f$ extends uniquely to a functor $\check{f}: \check{M}' \to \check{M}$, since the splittings of the idempotents of $M$ must be mapped to the splittings of their images, which forces $\check{f}(\underline{e}') := \underline{f(e')}$, and a morphism $m': \underline{e}' \to \underline{d}'$ must be sent to the conjugate of $f(m'): \underline{f(1')} \to \underline{f(1')}$ by the splitting components $\underline{f(e')} \hookrightarrow \underline{f(1')}$ and $\underline{f(1')} \too \underline{f(d')}$.

Similarly, $\check{\alpha}$ determines and is determined by $\alpha := \check{\alpha}_{\underline{1}'}$ because the horizontal morphisms in the naturality squares split:
\[\begin{tikzcd}
f(\underline{1}') \ar[r, two heads, shift left] \ar[d, "\alpha"]
& f(\underline{e}') \ar[d, "\check{\alpha}_{\underline{e}'}"] \ar[l, hook, shift left]\\
g(\underline{1}') \ar[r, two heads, shift left]
& g(\underline{e}') \ar[l, hook, shift left],
\end{tikzcd}\]
and $\alpha$ defined in this way is certainly a conjugation by the definition of the morphisms in $\check{M}$ and by the conditions imposed by the naturality square.

Finally, $\check{\alpha}$ is a natural isomorphism if and only if $\alpha$ is an isomorphism in $\check{M}$, which by inspection corresponds to the condition in Definition \ref{dfn:equiv}.
\end{proof}

By introducing 2-cells, we have constructed a 2-category $\Mon_s$ of monoids, semigroup homomorphisms between them, and conjugations between those. In this setting it is appropriate to explicitly state the relevant notion of equivalence imposed by the 2-cells.

\begin{dfn}
\label{dfn:equiv}
A semigroup homomorphism $f: M' \to M$ is an \textbf{equivalence} if there exists a homomorphism $g:M \to M'$, called its \textbf{pseudo-inverse}, along with invertible conjugations $\alpha: \id_{M'} \Rightarrow gf$ and $\beta: fg \Rightarrow \id_{M}$.
\end{dfn}

At first glance, this doesn't seem like a more general type of equivalence than isomorphism, but the weaker conditions on $\alpha$ and $\beta$ to be invertible make this a stronger type of equivalence in general. For a non-trivial example of this, see Example \ref{xmpl:Schein} below.

Let $\TOP^*_{\mathrm{ess}}$ be the 2-category whose objects are Grothendieck toposes having a surjective point (although this point need not be specified, since they are unique up to composition with autoequivalences), whose morphisms are essential geometric morphisms, and whose 2-cells are geometric transformations (natural transformations between the inverse image functors).

\begin{thm}
\label{thm:2equiv}
The functor $M \mapsto [M\op,\Set]$ is a 2-equivalence from $\Mon_s\co$ to $\TOP^*_{\mathrm{ess}}$.
\end{thm}
\begin{proof}
Directly, Proposition \ref{prop:extend} shows that the mapping $M \mapsto \check{M}$ is not only functorial but also full and faithful, and by \cite{Ele}[A4.1.5] the mapping $\Ccal \mapsto [\Ccal\op,\Set]$ is a full and faithful (but 2-cell reversing) functor from the sub-2-category of $\Cat$ on the idempotent-complete small categories to the 2-category of Grothendieck toposes, essential geometric morphisms and natural transformations. Therefore it suffices to show that the image of the composite is the stated subcategory.

That the composite lands inside $\TOP^*_{\mathrm{ess}}$ follows from the observations in Section \ref{sec:disc}. Conversely, given an object $\Ecal$ of $\TOP^*_{\mathrm{ess}}$, any essential surjective point provides an $M$ with $[M\op,\Set] \simeq \Ecal$ by Theorem \ref{thm:point}.
\end{proof}

This result can be compared directly with the 2-equivalence between the category $\Pos$ of posets, order-preserving functions and identity 2-cells and the corresponding 2-category of localic toposes with enough essential points, essential geometric morphisms between these and having geometric transformations as 2-cells, which arises as a consequence of the fact that posets are Cauchy complete. It can also be thought of as a first step towards a parallel of the results in section C1.4 of \cite{Ele} which gives a full equivalence of 2-categories between locales and localic toposes.

\begin{crly}
\label{crly:Morita}
Two monoids $M$ and $M'$ are Morita equivalent (that is, $[M\op,\Set] \simeq [{M'}\op, \Set]$) if and only if they are equivalent in the sense of Definition \ref{dfn:equiv}. This occurs if and only if there is an idempotent $e$ of $M$ with $M' \cong eMe$ and $\beta, \beta' \in M$ such that $\beta \beta' = 1$, $\beta e = \beta$.
\end{crly}
\begin{proof}
The first statement is a trivial consequence of Theorem \ref{thm:2equiv}, since all equivalences can be expressed as essential geometric morphisms. It remains to show that the stated data is sufficient to determine an equivalence in the sense of Definition \ref{dfn:equiv}.

First, given an equivalence $f:M' \to M$ (with pseudo-inverse $g:M \to M'$) the remaining data of the equivalence provides $e$ such that $\underline{e} = \underline{f(1')}$, $\beta$ and $\beta'$ with the given properties (amongst others!); to see that $M' \cong eMe$ one need only observe that the extension of $f$ to $\check{f}:\check{M}' \to \check{M}$, being an equivalence, must be full and faithful at $\underline{1}$, and so restricts to a bijective semigroup homomorphism.

Conversely, given $e, \beta, \beta'$ in $M$ with the given properties, note that replacing $\beta'$ with $e \beta'$ if necessary, one obtains elements with the additional property that $e \beta' = \beta'$. Let $M':= eMe$, let $f:M' \to M$ be its inclusion and consider the homomorphism $g: M \to M'$ given by $m \mapsto \beta' m \beta$. We see that $g$ is a semigroup homomorphism since $\beta' mn \beta = \beta' m \beta \beta' n \beta$, and it has the correct codomain since $e \beta' m \beta e = \beta' m \beta$.

Now $fg(1) = \beta' \beta$, so $\beta$ indeed constitutes an invertible conjugation $\id_M \rightarrow fg$; on the other side, since $gf(e) = \beta' e \beta$, taking $\alpha = e \beta$ and $\alpha' = \beta' e$ is easily seen to provide the other invertible conjugation to complete the equivalence.
\end{proof}

It is worth mentioning that the Morita equivalence presented here is distinct from the `Topos Morita Equivalence' for inverse semigroups discussed by Funk et al. in \cite{MEiS} (although the `Semigroup Morita equivalence' described there is the one introduced by Talwar in \cite{MES} based on the work of Knauer in \cite{MEM}). Indeed, the toposes considered there have as objects actions of an inverse semi-group $S$ on sets by \textit{partial isomorphisms}, which they show is equivalent to the topos of presheaves on the full subcategory $\check{S} \hookrightarrow \check{S}_1$ on the non-identity elements.

Rather than constructing a detailed example to demonstrate the distinction, we point out that the Morita equivalences of \cite{MEiS} are non-trivial, whereas the extension of Morita equivalence for monoids to semigroups described in Section \ref{sec:disc} is trivial by Corollary \ref{crly:trivial}.3 below, a fact which appears as Proposition 5 in \cite{FCM}.

\section{Examples and Corollaries}

To begin, here is an example demonstrating that Morita equivalence is (in general) strictly stronger than isomorphism.

\begin{xmpl}
\label{xmpl:Schein}
The `Schein monoids' were described by Knauer in \cite{MEM}. Consider the monoid $A$ of partial endomorphisms of $[0,1]$; that is, of those functions $A \to [0,1]$ where $A$ is some subset of $[0,1]$. The composite of two such morphisms $f: A \to [0,1]$ and $g:B \to [0,1]$ is defined to be the function $g \circ f: f^{-1}(B) \to [0,1]$.

Let $M$ be the submonoid of $A$ generated by the inclusions $e_x: [0,x] \hookrightarrow [0,1]$ for $3/4 \leq x \leq 1$, the halving map $\beta':[0,1] \to [0,1]$ sending $a \mapsto a/2$ and the doubling map $\beta:[0,1/2] \to [0,1]$ which is a left inverse to $\beta'$. By inspection $e_{3/4}, \beta, \beta'$ satisfy the required conditions to generate a Morita equivalence; let $M' = e_{3/4}Me_{3/4}$.

To see that $M$ and $M'$ are not isomorphic, observe that the idempotents of $M$ are all of the form $e_x$ for some $x \in [0,1]$; a more detailed case analysis demonstrates that the idempotents are precisely $e_x$ with $x \in [3/2^{n+2},1/2^n]$ for some $n \geq 0$. The idempotents come with a canonical order given by $e_x < e_y$ if $x<y$, or equivalently if $e_x e_y = e_x$; this order is thus preserved by isomorphism. The non-identity idempotents of $M$ have no maximal element. However, the non-identity idempotents of $M'$ do have a maximum (specifically $e_{1/2}$). Thus $M \not\cong M'$. 
\end{xmpl}

This and further examples are collected in \cite{MAC}. It should be clear, however, that the conditions in Corollary \ref{crly:Morita} force Morita equivalence to reduce to isomorphism in many important cases.

\begin{crly}
\label{crly:trivial}
Let $M$ be a monoid. Then for equivalence to coincide with isomorphism at $M$, any of the following conditions suffices:
\begin{enumerate}
	\item $M$ is commutative.
	\item $M$ is a group.
	\item Every right (or every left) invertible element of $M$ is invertible; equivalently, the non-units of $M$ are closed under multiplication (such as when $M=S_1$ for a semigroup $S$).
	\item $M$ is left (or right) cancellative.
	\item The idempotents of $M$ satisfy the descending chain condition with respect to absorption on the right (or left).
	\item The left (or right) ideals of $M$ satisfy the descending chain condition.
\end{enumerate}
\end{crly}
\begin{proof}
It suffices to examine the condition for equivalence in Corollary \ref{crly:Morita}. We obtain an equivalence with $M$ whenever $M$ contains elements $\beta$, $\beta'$ and an idempotent $e$ with $\beta \beta' = 1$ and $\beta e = \beta$ (the equivalence is with $eMe$); if such a $\beta$ is necessarily an isomorphism, this forces $e=1$, so the Morita equivalence class is trivial and the equivalence collapses to an inner automorphism of $M$. In the first three cases, the equation $\beta \beta' = 1$ indeed forces $\beta$ to be an isomorphism, while in the fourth case $\beta e = \beta$ forces $e=1$ so there is nothing further to do.

For the last two conditions, note that $e_n:= {\beta'}^n {\beta}^n$ is an idempotent for every $n$, with the property that $e_n e_m = e_n = e_m e_n$ whenever $n\geq m$; if it is ever the case that $e_{n+1} = e_n$, then by multiplying on the left by ${\beta}^n$ and on the right by ${\beta'}^n$ it is again the case that $\beta' \beta = 1$. Thus for equivalence to be non-trivial $M$ must have an infinite descending chain of idempotents. By instead considering the ideals $Me_n$ we reach a similar conclusion for ideals.
\end{proof}

These conditions are variants of those which appear in \cite{MEM} and \cite{FCM}. They can also be interpreted as properties of sites which are invariant under Morita equivalence. Any such property necessarily has \textbf{corresponding invariants at the topos-theoretic level}. If these can be identified, each gives its own immediate Corollary of Theorem \ref{thm:2equiv}. For example:

\begin{crly}
The mapping $G \mapsto [G\op,\Set]$ is an equivalence between the 2-category $\Grp \simeq \Grp\co$ of groups, group homomorphisms and conjugations and the 2-category $\TOP^*_{\mathrm{at, ess}}$ of \textbf{atomic} Grothendieck toposes with an essential surjective point, essential geometric morphisms and natural transformations.
\end{crly}
\begin{proof}
Note that any semigroup homomorphism between groups is automatically a group homomorphism. Thus this equivalence is simply a restriction of the earlier one, and it suffices to show that the essential image is what we claim it is. To see that any topos of the form $[G\op,\Set]$ is atomic it suffices to observe that the indecomposable $G$-sets can be identified with quotients of $G$ by subgroups, and each such is an atom (has no non-trivial sub-$G$-sets). See \cite{TGT} for a more general and detailed argument.

Conversely, if $[M\op,\Set]$ is atomic, consider the action of $M$ on itself by right multiplication, which is indecomposable by transitivity. Given any $m \in M$, it must be that $Mm = M$, else $Mm$ would be a non-trivial sub-$M$-set. Thus $1 \in Mm$ and $m$ is left invertible, whence every element of $M$ is a unit and $M$ is a group, as required.
\end{proof}

\bibliographystyle{plain}
\bibliography{classificationbib}

\end{document}